\documentclass[12pt, twoside, leqno]{article}

\usepackage{amsmath,amsthm}
\usepackage{amssymb}

\usepackage{enumitem}

\usepackage{graphicx}

\usepackage[T1]{fontenc}

\newtheorem{theorem}{Theorem}[section]

\newtheorem{lemma}[theorem]{Lemma}

\theoremstyle{definition}

\newtheorem{remark}[theorem]{Remark}

\numberwithin{equation}{section}

\newcommand{\Q}{{\mathbb {Q}}}

\newcommand{\R}{{\mathbb{R}}}

\newcommand{\Z}{{\mathbb{Z}}}

\newcommand{\N}{{\mathbb{N}}}

\newcommand{\dist}{{\rm dist}}

\newtheorem{innercustomthm}{Theorem}
\newenvironment{customthm}[1]
  {\renewcommand\theinnercustomthm{#1}\innercustomthm}
  {\endinnercustomthm}

\newcommand\da{Diophantine approximation}

\newcommand {\ignore}[1] {}

\newcommand\ssm{\smallsetminus}

\newtheorem*{remark*}{Remark}

\begin{document}


\baselineskip=17pt


\title{Simultaneous Diophantine approximation:  
sums of squares and homogeneous polynomials}

\author{Dmitry Kleinbock\\
Brandeis University\\ 
Waltham MA, USA 02454-9110\\
E-mail: kleinboc@brandeis.edu
\and 
Nikolay Moshchevitin\\
Moscow State University\\ 
Lenlnskie Gory 1, Moscow, Russia,   1119991\\
and Astrakhan State University\\
Tatishcheva 20a\\
Astrakhan, Russia, 414056\\
E-mail: moshchevitin@gmail.com}

\date{}

\maketitle


\renewcommand{\thefootnote}{}



\renewcommand{\thefootnote}{\arabic{footnote}}
\setcounter{footnote}{0}


\begin{abstract}
 Let {$f
$ be a homogeneous polynomial with rational coefficients in $d$ variables}.
We prove several results concerning uniform simultaneous approximation to {points on the graph of $f$, as well as   on the hypersurface  $\{f(x_1,\dots,x_d) = 1\}$.
The results are first stated for the case $f(x_1,\dots,x_d) = x_1^2+\dots+x_d^2,$ which  is of  particular interest.}
 \end{abstract}

\section{Diophantine exponents}
 Let $\Theta = (\theta_1,\dots,\theta_m)$ be a collection of real numbers. 
 The {\sl ordinary Diophantine exponent} $\omega = \omega(\Theta)$ 
 for simultaneous  {rational} approximation to $\Theta$ 
 is defined as the supremum over all real $\gamma$ such that 
 the inequality 
 $$
 \max_{1\le j \le m} |q\theta_j - a_j| < q^{-\gamma}
 $$
 has infinitely many solutions  in integer points $(q,a_1,\dots,a_m)\in \mathbb{Z}^{m+1}$
 with $ q >0$.
 
 The {\sl uniform 
 Diophantine exponent} $\hat{\omega} = \hat{\omega}(\Theta)$ 
 for simultaneous approximation to $\Theta$ 
 is defined as the supremum over all real $\gamma$ such that 
 the system of inequalities 
 $$
 \max_{1\le j \le m} |q\theta_j - a_j| < Q^{-\gamma},\,\,\,\, 1\le q \le Q
 $$
 has a solution $(q,a_1,\dots,a_m)\in \mathbb{Z}^{m+1}$ for every large enough real $Q$.
 It {immediately follows from Minkowski's convex body  theorem that $\hat{\omega}(\Theta)\ge \frac{1}{m}$ for any $\Theta\in\R^m$. Furthermore, let us say that $\Theta$ is
{\sl totally irrational}
if $1, \theta_1,\dots,\theta_m$
are linearly
independent over
$\Z$. For such $\Theta$} it
 was first observed by  Jarn\'ik   \cite[Satz 9]{JJ} that 
 $$
 \hat{\omega}(\Theta)\le 1.
 $$
 (See also
 \cite[Theorem 17]{MK},  {as well as}
  \cite[Theorem 5.2]{W} for a proof based on homogeneous dynamics.)
 In particular for $m=1$  one has 
 \begin{equation}\label{odin}
 \hat{\omega}(\theta )=1\,\, {\text{ for all } \theta \in \mathbb{R}\ssm\Q}.
 \end{equation}
{On the other hand,} for $m\ge 2$ {it is known that}  for arbitrary $ \lambda$ from the interval $\left[\frac{1}{m},1\right]$ there exists ${\Theta\in\R^m}$ with
 $\hat{\omega}(\Theta) = \lambda$.
 
 
 Moreover
 it is clear from the definition that 
 $$ 
 {\omega}(\Theta) \ge \hat{\omega}(\Theta) 
$$
{for any $\Theta \in\R^m$.} Here we should mention that in \cite{J}  Jarn\'ik gave an improvement of this bound for the collection {of} $\Theta$ 
{such that} there are at least two numbers $\theta_i, \theta_j$ linearly independent over $\mathbb{Z}$ together with $1$.
In this case he proved the inequality
$$
\frac{\omega}{\hat{\omega}}\ge \frac{\hat{\omega}}{1-\hat{\omega}}.
$$
 This inequality is optimal for $m=2$.
 For arbitrary $m$ the optimal inequality
 was obtained recently 
  by
  Marnat and Moshchevitin \cite{MM}.

\begin{customthm}{A}\label{A} {\cite[Theorem 1]{MM}}
 Let {$\Theta\in\R^m$  
   { be totally irrational}, and let  ${\omega} = {\omega}(\Theta)$ and  $\hat{\omega} = \hat{\omega}(\Theta)$. Denote by   $G_{m}$ the unique positive root of the equation
   \begin{equation}\label{root}
   x^{m-1} = \frac{\hat{\omega}}{1-\hat{\omega}} (x^{m-2} + x^{m-3} + \dots+ x+1).
\end{equation}
   Then} 
 one has
   \begin{equation*}\label{pp}
  \frac{\omega}{\hat{\omega}}\ge {G_{m}}.
   \end{equation*}
\end{customthm}

   In the present paper we study the bounds for the uniform exponent  $\hat{\omega}$ for special collections of numbers.
   Theorem \ref{A} will be {an} important ingredient of our proofs.
 
\section{Approximation to several real numbers and sums of their squares}\label{appr}

 In \cite{DS} Davenport and Schmidt proved the following theorem:

\begin{customthm}{B}\label{B}  {\cite[Theorem 1a]{DS}} Suppose that  $\xi\in \R$ is neither a rational number  nor a quadratic irrationality.
Then the uniform Diophantine exponent $\hat{\omega} = \hat{\omega}(\Xi)$ {of} the vector $ \Xi = (\xi, \xi^2)\in \mathbb{R}^2$ 
 satisfies the inequality
$$
\hat{\omega} \le \frac{\sqrt{5}-1}{2}.
$$
\end{customthm}
 
 Here we should note that $\frac{\sqrt{5}-1}{2}$
 is the unique positive root of the equation
 $$
 x^2 +x =1.
 $$
  It is known due to Roy \cite{R} that the bound of Theorem \ref{B} is optimal.
  Davenport and Schmidt  proved a more general result  {\cite[Theorem 2a]{DS}}  
  {involving} successive powers $ \xi,\xi^2, \dots,\xi^m$.
  However in the present paper we deal with  another generalization.

 In the sequel we will consider 
$m = d$ or $m = d+1$ numbers.
 Namely, 
{take  $$\pmb{\xi} = (\xi_1,\dots,\xi_d)\in\mathbb{R}^d$$ and introduce  the vector} 
\begin{equation}\label{oooo}
  \Xi = (\xi_1,\dots,\xi_d, \xi_1^2+\dots+\xi_d^2) \in\mathbb{R}^{d+1}.
   \end{equation}
  {Also let} $H_d$ be the unique positive root of the equation
  \begin{equation}\label{po}
x^{d+1} + x^{d} + \dots+ x  = 1.
\end{equation}
  {Note that}
  %
$ \frac{1}{2}< H_d < 1$, and $ H_d \to \frac{1}{2}$  monotonically when $d \to \infty$.

  \medskip
  In the present paper we prove {the following  two} theorems dealing with sums of squares.
 
\begin{theorem}\label{1} Let $ d\ge 1$ be an integer.
Suppose that {$\Xi $ {as in \eqref{oooo}} is totally irrational.}
Then the {uniform} Diophantine exponent 
{of $\Xi $}
satisfies the inequality
\begin{equation*}\label{ff}
\hat{\omega} ({\Xi})\le H_d.
\end{equation*}
\end{theorem}

{Note that} {in}  the case $d = 1$ 
 Theorem \ref{1}
coincides with {Theorem \ref{B}}. 
 The {next} theorem {can} be proved by {a} similar argument.

\begin{theorem}\label{2} Let $d\ge 2$.
Suppose that  
{$\pmb{\xi} = (\xi_1,\dots,\xi_d)$ is totally irrational} and
 \begin{equation}\label{0}
 \xi_1^2+\dots+\xi_d^2= 1.
\end{equation}
Then the  {uniform} Diophantine exponent 
{of} $\pmb{\xi} $
satisfies the inequality
\begin{equation*}\label{f}
\hat{\omega} (\pmb{\xi})\le H_{d-1}.
\end{equation*}
\end{theorem}
 Theorems \ref{1} and \ref{2} are particular cases of more general Theorems \ref{1a} and \ref{2a},  which we formulate in Section \ref{N}. 

{\begin{remark}
It is worth comparing Theorems \ref{1} and \ref{2}, as well as their more general versions, with {a lower bound obtained using the methods of \cite[\S 5]{KW}.  It is not hard to derive from \cite[Corollary 5.2]{KW} that for any 
 real analytic 
submanifold $M$ of $\R^m$ of 
dimension at least $2$  which is not contained in any proper rational affine
hyperplane of $\R^m$  there exists totally irrational   $\Theta\in M$ with $$\hat\omega(\Theta)  \ge \frac1m +  \frac2{m(m^2-1)}.$$
It would be interesting to see if the above estimate could be improved,} thus shedding some light on the optimality of our theorems.
\end{remark}}

\section{Intrinsic approximation on spheres}\label{intr}

Our study of  vectors of the form (\ref{oooo})
was motivated by problems related to intrinsic rational approximation  on spheres.
In \cite{K} Kleinbock and  Merrill proved the following result.

\begin{customthm}{C}\label{C}   {\cite[Theorem 4.1]{K}} Let $ d\ge 2$.
There exists a positive constant
 $C_d$ such that for any 
$\pmb{\xi} = (\xi_1, \dots,\xi_n) \in \mathbb{R}^d$ 
{satisfying} \eqref{0} and for any $T>1$ there exists 
a rational vector
\begin{equation}\label{uuu}
\pmb{\alpha} = \left( \frac{a_1}{q},\dots,    \frac{a_d}{q}\right)\in  \mathbb{Q}^d
\end{equation}
such that
$$
\pmb{|} \pmb{\alpha} \pmb{|}^2 =
\left( \frac{a_1}{q}\right)^2+ \dots + \left(\frac{a_d}{q}\right)^2 = 1
$$
and
\begin{equation}\label{kle}
\pmb{|} \pmb{\xi} - \pmb{\alpha}\pmb{|}\le \frac{C_d}{q^{1/2}T^{1/2}}
,\,\,\,\,\,\,
1\le q\le T.
\end{equation}\end{customthm}

{Here and hereafter by}
  $\pmb{|} \cdot \pmb{|}$
   we denote the Euclidean norm of a vector.
In particular Theorem~\ref{C} {implies} that in the case  $\pmb{\xi} \not \in \mathbb{Q}^d$ the inequality
$$
\pmb{|} \pmb{\xi} - \pmb{\alpha}\pmb{|}\le \frac{C_d}{q}
$$
has infinitely many solutions in rational vectors (\ref{uuu}).

{See \cite{M,M1} for effective versions of Theorem \ref{C}, and  \cite[Theorem 5.1]{F}  for {generalizations}. 
{Note that  the formulation from \cite{M} involves   sums of squares, while an effective  version for  an arbitrary positive definite quadratic form  with integer coefficients can be found in \cite{M1}.}
It is also explained in \cite{F} how  the conclusion of Theorem \ref{C} can be derived from \cite[Theorem 1]{SV} via a correspondence between   intrinsic
Diophantine approximation on quadric hypersurfaces and   approximation
of points in the boundary of the
 hyperbolic space
by parabolic fixed points of Kleinian groups; see  \cite[Proposition  3.16]{F}.}

\smallskip
{In this paper} we prove  a result about uniform intrinsic approximation on the {unit} sphere.
We need some notation. First of all, {note that} the inequality (\ref{kle}) can be rewritten as
$$
\frac{\pmb{|}q \pmb{\xi} - \pmb{a}\pmb{|}^2}{q}
\le \frac{C_d^2}{T},\,\,\,\,  \pmb{a} = (a_1,\dots,a_d) \in \mathbb{Z}^d, \,\,\,\, a_1^2+\dots+a_d^2 = q^2.
$$
Now {let us} define the function
$$
\Psi_{\pmb{\xi}} (T) =
\min_{(q, a_1,\dots,a_d)\in \mathbb{Z}^{d+1}:  \, 1\le  q\le T , \,\, a_1^2+\dots+a_d^2 = q^2} \,\,
\frac{  
 \pmb{|}q \pmb{\xi} - \pmb{a}\pmb{|}^2
}{q}.
$$
Theorem  \ref{C}  states that for any $\pmb{\xi} \in \mathbb{R}^d$ under the condition  (\ref{0}) one has
$$
T\cdot  \Psi_{\pmb{\xi}} (T) \le C_d^2\,\,\,\,\,\text{for}\,\,\,\,\, T >1.
$$

\begin{theorem}\label{3} Let $ d\ge 2$. Let $\pmb{\xi} \in\R^d\ssm \mathbb{Q}^d$  {be such that \eqref{0} is satisfied}.
Then  for any   $\varepsilon > 0$ there exists arbitrary large $T$
 such that 
 $$
 T\cdot  \Psi_{\pmb{\xi}} (T) \ge \frac{1}{4} -\varepsilon.
 $$
\end{theorem}

 Theorem \ref{3}  is an analog of Khintchine's lemma on rational approximations to one real number
   (see   \cite[Satz 1]{Khi}).
 It admits the following {corollary}. 
  One can try to define the {\sl  uniform Diophantine exponent}  {of} $\pmb{\xi} $ 
  for the intrinsic approximation on {the unit} sphere
 as
 $$
{ \hat{\omega}_d^{\bf i}  (\pmb{\xi})
 = \sup
\left \{\gamma \in \mathbb{R} \left|\begin{aligned}\text{ the inequalities} \,\,\,\,\,\,\,
\pmb{|} \pmb{\xi} - \pmb{\alpha}\pmb{|}\le \frac{1}{q^{1/2}T^{\gamma/2}},\,\,\,\,\,
1\le q\le T\quad  \\
\text{ are solvable in  
}   \pmb{\alpha} 
\text{ of the form (\ref{uuu}) for 
large enough } T \end{aligned}\right.\right\}.}
$$
Then  for all vectors  $\pmb{\xi}  \not\in \mathbb{Q}^d$ 
 {satisfying} (\ref{0})  we have 
 $$
  \hat{\omega}_d^{\bf i}  (\pmb{\xi})=1
 $$
 by Theorem \ref{3}.
 So here we have an equality similar to (\ref{odin}) for  the case of approximation to one real number. {See also \cite[Theorem 2]{BGSV}, where {a similar} observation was made in the context of Kleinian groups.}
  Theorem~\ref{3}  follows from a more general Theorem \ref{3a}
   which we formulate in Section~\ref{N}.


\section{{Results on homogeneous polynomials}} \label{N}

{Given  integers $s \ge 2$ and $d\ge 1$,} define $H_{d,s}$ {to be} the unique positive root of the equation
{\begin{equation}\label{0qqq}
 (1-x) =x\cdot \sum_{k=1}^{d} \left(\frac{x}{s-1}\right)^k.
\end{equation}
{Note that for any $s$ and $d$ one has}
$$
\frac{s-1}{s} < H_{d,s} <1.
$$
{Clearly}  $H_{d,2} = H_d$,  and $ H_{d,s}  $ monotonically decreases to $\frac{s-1}{s}$ as $ d \to +\infty$.

The results of this section deal with
 a homogeneous polynomial 
\begin{equation}\label{polyh}
f(\pmb{x} ) =
\sum_{(s_1,\dots,s_d) \in \mathbb{Z}^d_+:\,\,
s_1+\dots+s_d = s}
f_{s_1,\dots,s_d} x_1^{s_1}\cdots x_d^{s_d},\quad{\text{where }f_{s_1,\dots,s_d}\in \mathbb{Q},}
\end{equation}
of degree $s$ in  variables $x_1,\dots,x_d$
(here $\mathbb{Z}_+$ stands for the set of non-negative integers).
Theorem \ref{1} from the previous section is a corollary of the following general statement.

\begin{customthm}{1a}\label{1a} 
Let $s\ge 2$ be an integer,  {and let $f$ as in \eqref{polyh} be such that 
\begin{equation}\label{rank1}
{\#\{\pmb{x}\in\Q^d: f(\pmb{x} )= 0\} < \infty}
.\end{equation}
Suppose that}
\begin{equation}\label{ooooW}
\Xi_f = \big(\xi_1,\dots,\xi_d, f(\xi_1,\dots,\xi_d)\big)
\end{equation}
{is totally irrational.} Then $\hat{\omega}({\Xi_f})\le H_{d,s}$.
\end{customthm}

We give a proof of Theorem \ref{1a} in Section \ref{t4}.

\begin{remark}\label{r2}
Let us consider the case $  d=1$. In this case Theorem \ref{1a} states that the uniform exponent $\hat{\omega}$ of 
$(\xi, \xi^s)$ 
is bounded from 
{above} by the positive root of the equation
$$
x^2 +(s-1)x-(s-1) = 0,
$$
that is
\begin{equation}\label{ba}
\hat{\omega} \le \frac{\sqrt{(s-1)(s+3)}-(s-1)}{2}.
\end{equation}
This result  was obtained by  Batzaya in \cite{B1}
for arbitrary vectors of the form $(\xi^l,\xi^s)$ with $1\le l < s$. 
In the case  $d=1 ,s=3$ much 
stronger inequality
$$
\hat{\omega} \le  \frac{2(9 + \sqrt{11})}{35}
$$
is known due to Lozier and Roy (see \cite{LR} and the discussion therein).
In \cite{B2} Batzaya  improved (\ref{ba}) and showed that
for   $(\xi^l,\xi^s)$ with $1\le l < s$ one has
$$
\hat{\omega} \le  
\frac{s^2-1}{s^2-s-1}
$$
for odd $s$. In the case of even $s$  in the paper \cite{B1} he had  a better inequality
$$
\hat{\omega} \le  
\frac{(s-1)(s+2)}{s^2+2s-1}.
$$
Also   \cite{B2} contains {a} better bound for $
\hat{\omega}$
when $ s = 5,7,9$.
Thus the inequality of our Theorem 1a is not optimal for $s \ge 3$.
\end{remark}


 
Theorem \ref{2} from the previous section is a corollary of the following general statement.

\begin{customthm}{2a}\label{2a} Let $s\ge 2$ be an integer,  {and let $f$ as in \eqref{polyh} be such that \eqref{rank1} holds.}
Then 
{\begin{equation}\label{conclusion}
\pmb{\xi}\in \mathbb{R}^d\text{ is totally irrational and } f(\pmb{\xi}) = 1\quad\Longrightarrow\quad
\hat{\omega}(\pmb{\xi})\le H_{d-1,s}.
\end{equation}}
\end{customthm}

We give a proof of Theorem \ref{2a} in Section \ref{t5}.
To get Theorems  \ref{1}, \ref{2} from Theorems \ref{1a}, \ref{2a} one should put $ s= 2$ and $f(\pmb{x}) = x_1^2+\dots+x_d^2$.

{\begin{remark}\label{r3} The argument used in the proof of Theorem \ref{2a} yields \eqref{conclusion} for any (not necessarily homogeneous) polynomial  $f$   with rational coefficients 
 such that the number of rational points on the hypersurface $\{f = 1\}$ is finite. We state it as Theorem \ref{2b} in Section  \ref{t5}. For example (cf.~\eqref{ba} with $d=1$ and $s = 6$)  it follows that $$\hat{\omega}(x,y) \le \frac{\sqrt{45}-5}2$$  for any $(x,y)\in\R^2$ such that $y^2 - x^2 - x^6 = 1$.\end{remark}
}
 
Now for $\pmb{\xi}\in \mathbb{R}^d$ under the condition $
f(\pmb{\xi }) = 1
$
consider the function
$$
\Psi_{f,\pmb{\xi}} (T) =
\min_{( q,\pmb{a}) = (q,a_1,\dots,a_d)\in \mathbb{Z}^{d+1}:\,\,
1\le q\le T,\,\,  f\left(\frac{a_1}{q},\dots,\frac{a_d}{q}\right) = 1}\,\,\,\,\,
\frac{\pmb{|}   q\pmb{\xi}- \pmb{a} \pmb{|}^s}{q}.
$$
 It is clear that $
\Psi_{f,\pmb{\xi}} (T) $ is a non-increasing piecewise constant function. Here we do not suppose that it tends to zero as $T \to +\infty$.

\begin{customthm}{3a}\label{3a} 
 Let $s\ge 2$ be an integer,  {and let $f$ as in \eqref{polyh} be such that \eqref{rank1} holds.
Take} 
$\pmb{\xi} \not\in \mathbb{Q}^d$ {with}
$$
f(\pmb{\xi }) = 1,
$$
  and
let $D = D(f)\in \mathbb{Z}_+$ be the common denominator of all rational numbers $f_{s_1,\dots,s_d}$.
{Also define}
\begin{equation}\label{20qqq}
{ K} =K(f) = \sup_{\pmb{x}\in\mathbb{R}^d: \,\, \pmb{|}   \pmb{x} \pmb{|} = 1}
|f(\pmb{x})|.
\end{equation}
Then  for any positive $\varepsilon$ there exists arbitrary large $T$
 such that 
 $$
 T^{s-1}\cdot  \Psi_{f,\pmb{\xi}} (T) \ge \frac{1}{2^sD{ K}  } -\varepsilon.
 $$
 \end{customthm}

For $ f (\pmb{x}) = x_1^2+\dots+x_d^2$ we have $ s = 2$ {and} $D(f) = K(f) = 1$.
{Thus} Theorem \ref{3} is a direct corollary of Theorem \ref{3a}.
We give a proof of Theorem \ref{3a} in Section \ref{t6}.

\section{{The main lemma}} \label{lem}

{The next lemma is a 
{polynomial} analogue of the classical simplex lemma {in simultaneous \da\ going back to Davenport  \cite{D}}. See {also} \cite{KS} for a {version for} {arbitrary quadratic forms},   \cite[Lemma 1]{BGSV} for a similar statement in the context of Kleinian groups, and  \cite[Lemma 4.1]{fkmsgeneral}  for a general simplex lemma for intrinsic Diophantine approximation on manifolds.}

\begin{lemma}\label{l1}
Let
$s\ge 2 $ be an integer,  {and let $f$ be as in \eqref{polyh}.}
Let $D =D(f)$ and $K=K(f)$ be defined as in Theorem \ref{3a}, {and take two rational vectors}
$$\pmb{\alpha} = \left( \frac{a_1}{q},\dots,    \frac{a_d}{q}\right)
{\text{ and \ }}
\pmb{\beta} = \left( \frac{b_1}{r},\dots,    \frac{b_d}{r}\right) $$
{such} that 
\begin{equation}\label{zzz}
f(\pmb{\alpha}-\pmb{\beta} )\neq 0.
\end{equation}
   
    \noindent {\bf (i)}
   Suppose that
   $$
  f(\pmb{\alpha}) = \frac{A}{q}
$$
with an integer $A$.
   Then
  \begin{equation}\label{oo}
   \pmb{|} \pmb{\alpha} - \pmb{\beta}\pmb{|}^s
   \ge 
   \frac{1}{DKq^{s-1} r^s}\,.
   \end{equation}
   
      \noindent {\bf (ii)}
      Suppose  
$$
  f(\pmb{\alpha})=\frac{A}{q},\,\,\,
  f(\pmb{\beta}) = \frac{B}{r}
$$
with integers $A,B$.
   Then
 \begin{equation}\label{ooooo}
   \pmb{|} \pmb{\alpha} - \pmb{\beta}\pmb{|}^s
   \ge 
\frac{1}{DKq^{s-1} r^{s-1}}\,.
\end{equation}

   \end{lemma}
   
\begin{proof}

\noindent {\bf (i)}
 \, First of all  
we observe   that
 \begin{equation}\label{1qqq}
 | f(\pmb{\alpha} -\pmb{\beta})|\ge \frac{1}{Dq^{s-1} r^s}.
 \end{equation}
 Indeed, 
 for any $s_1,\dots,s_d \in \mathbb{Z}_+$ under the condition $s_1+\dots+s_d = s$
 consider the product
 $$
 \Pi_{s_1,\dots,s_d} = \prod_{k=1}^d 
 \left(\frac{a_k}{q}-\frac{b_k}{r}\right)^{s_k}.
 $$
 It is clear that
 $$
  \Pi_{s_1,\dots,s_d}  = \frac{\prod_{k=1}^d a_k^{s_k}}{q^s}+ \frac{W_{s_1,\dots,s_d}}{q^{s-1}r^s}
  $$
  with an integer $W_{s_1,\dots,s_d}$. Now from (\ref{zzz}) we see that
 \begin{equation*}
 \begin{aligned}
  0\neq f(\pmb{\alpha} -\pmb{\beta})
&=
\sum_{(s_1,\dots,s_d) \in \mathbb{Z}^d_+:\,\,
s_1+\dots+s_d = s}
f_{s_1,\dots,s_d}
 \Pi_{s_1,\dots,s_d} \\ &=
 f(\pmb{\alpha} )  + 
 \frac{W}{Dq^{s-1}r^s} = \frac{A}{q}  + 
 \frac{W}{Dq^{s-1}r^s} =
  \frac{W_1}{Dq^{s-1}r^s} 
 ,
 \end{aligned}
 \end{equation*}
 with $W, W_1\in \mathbb{Z}$,
 and (\ref{1qqq}) is proved.
Then from the definition (\ref{20qqq}) we see that
  \begin{equation}\label{201qqq}
|f(\pmb{\alpha}-\pmb{\beta})| \le  K \pmb{|} \pmb{\alpha}-\pmb{\beta}\pmb{|}^s .
\end{equation}
Now (\ref{1qqq}) and (\ref{201qqq}) give (\ref{oo}).
 
 \smallskip
 
   \noindent
   {\bf (ii)}\, The proof here is quite similar. From the conditions on
   $f(\pmb{\alpha})$ and
 $f(\pmb{\beta})$ we see that
  $$
  0\neq f(\pmb{\alpha} -\pmb{\beta})
  =
 f(\pmb{\alpha} )  \pm  f(\pmb{\beta} ) +
 \frac{W'}{Dq^{s-1}r^{s-1}} =  \frac{A}{q}    \pm  \frac{B}{r}  +
 \frac{W'}{Dq^{s-1}r^{s-1}}=  \frac{W_1'}{Dq^{s-1}r^{s-1}}
 ,
 $$
 with $W', W_1'\in \mathbb{Z}$. So we get
  \begin{equation}\label{13qqq}
 | f(\pmb{\alpha} -\pmb{\beta})|\ge \frac{1}{Dq^{s-1} r^{s-1}}.
 \end{equation}
Now (\ref{13qqq}) together with  (\ref{201qqq}) give (\ref{ooooo}).
 
 \end{proof}


\section{Best approximation vectors 
}\label{best}
 {If $\Theta = (\theta_1,\dots,\theta_m)\in\R^m\ssm\Q^m$, 
let us say that a vector $
   (q, a_{1},\dots,a_{m}) \in \mathbb{Z}^{m+1}$ is a {\sl best simultaneous approximation vector} of $\Theta$ if 
   $$
 \dist(q\Theta, \Z^m) <   \dist(k\Theta, \Z^m)\quad  \forall\, k = 1,\dots,q-1,
 $$
 and 
 $$
   \dist(q\Theta, \Z^m) = \max_{i = 1,\dots,m}   |q \theta_i - a_{i}| .
 $$   
 Here `dist' stands for the distance induced by the supremum norm on $\R^m$. Best approximation {vectors} of $\Theta$   form an infinite sequence $ (q_\nu, a_{1,\nu},\dots,a_{m,\nu})$, $\nu\in\N$,  and satisfy  the inequalities
 $$
   q_{\nu-1} < q_\nu,\,\,\,\,\,\,\,
   \zeta_{\nu-1} > \zeta_\nu, \,\,\,\,\,\,\,
  {\nu\in\N},
   $$  
   where one defines
 $$   \zeta_\nu =  \max_{i = 1,\dots,m}   |q_\nu \theta_i - a_{i,\nu}|.$$} 
 It is important that
 $$
 {\rm g.c.d. } (q_\nu, a_{1,\nu},\dots,a_{m,\nu}) = 1,\,\,\,\,\,
   {\nu\in\N}.
   $$
   So for any two successive rational approximation vectors
   $$
   \pmb{\alpha}_j =
   \left(
   \frac{a_{1,j}}{q_j},\dots,
   \frac{a_{m,j}}{q_j}
   \right) \in \mathbb{Q}^m,\,\,\,\, j = \nu-1, \nu
   $$
   we have
   $$
   \pmb{\alpha}_{\nu-1} \neq \pmb{\alpha}_\nu.
   $$

Some detailed information about best approximation vectors may be found for example in papers
   \cite{Ch} and \cite{MK}.  {In particular,} the following property of the uniform exponent $\hat{\omega} = \hat{\omega} ({\Theta})$ is well known (see {e.g.}\   \cite[Proposition 1]{MK}).
    Suppose that $ \gamma <\hat{\omega}$.
   Then for all $\nu $ large enough one has
   \begin{equation}\label{ww}
   \zeta_{\nu-1} \le q_\nu^{-\gamma}.
   \end{equation}

    \section{ Proof of Theorem \ref{2a}}\label{t5}

    We take $m=d$ and consider  best approximation vectors  
   $$
   \pmb{z}_\nu =( q_\nu, a_{1,\nu},\dots, a_{d,\nu})\in \mathbb{Z}^{d+1}
   $$
     {of} 
     $\pmb{\xi}  = (\xi_1,\dots,\xi_d)\in\mathbb{R}^d$, {together with  distances}
   $$
   \zeta_\nu = \max_{1\le j \le d} |q_\nu\xi_j - a_{j,\nu}|,
   $$
   and the corresponding rational approximants
   $$
   \pmb{\alpha}_\nu =\left(  \frac{a_{1,\nu}}{q_\nu},\dots,\frac{ a_{d,\nu}}{q_\nu}\right)\in \mathbb{Q}^{d}.
   $$
   Under the condition $\gamma < \hat{\omega} (\pmb{\xi})$  we have (\ref{ww})
  for all large $\nu$.
  
  Here we should note that
  $$
  \max_{1\le j \le d} \left|\xi_j - \frac{a_{j,\nu}}{q_\nu}\right| =
  \frac{\zeta_\nu}{q_\nu}.
  $$
So for large $\nu$ we see  that

\begin{equation}\label{6qqq}
 \left(\frac{a_{1,\nu }}{q_{\nu}}\right)^{s_1}  \cdots\ 
\left(\frac{a_{d,\nu-1}}{q_{\nu}}\right)^{s_d} = \xi_1^{s_1}\dots\xi_d^{s_d} + O\left(
\frac{\zeta_\nu}{q_\nu}\right).
\end{equation}

We consider two cases.

\medskip
\noindent {\bf Case 1.}  $ f (\pmb{\alpha}_\nu )=1 $ for infinitely many $\nu$.
\smallskip

Here,  {since} 
$\pmb{\alpha}_\nu\neq\pmb{\alpha}_{\nu-1}$,
we may apply Lemma \ref{1}{\bf (i)}
with $ A=q_\nu$.
{Take $\pmb{\alpha} = \pmb{\alpha}_\nu$, $\pmb{\beta}=\pmb{\alpha}_{\nu-1}$; then \eqref{zzz} follows from \eqref{rank1} when $\nu$ is large enough, and 
from} (\ref{oo})     we deduce that
\begin{equation}\label{c1}
\begin{aligned}
 \frac{1}{(DK)^{\frac{1}{s}}q_{\nu}^{\frac{s-1}{s} }q_{\nu-1}}
 &\le
 \sqrt{d}  \max_{1\le k\le d}
 \left|
 \frac{a_{k,\nu}}{q_\nu}-
  \frac{a_{k,\nu-1}}{q_{\nu-1}}
 \right|
\\& \le
 \sqrt{d}  
  \left(
  \max_{1\le k\le d}
 \left|
 \frac{a_{k,\nu}}{q_\nu}-\xi_k
 \right|+
  \max_{1\le k\le d}
 \left|
  \frac{a_{k,\nu-1}}{q_{\nu-1}}
 -\xi_k
 \right|
 \right)\\
& \le
\frac{2 \sqrt{d}  \zeta_{\nu-1}}{q_{\nu-1}}
\le
\frac{2 \sqrt{d}  }{q_{\nu-1}q_\nu^\gamma}\,.
\end{aligned}
 \end{equation}
 So  with some positive $c_1$ we have $ q_\nu^\gamma \le c_1 q_\nu^{\frac{s-1}{s}}$ for infinitely many $\nu$ and
 $\hat{\omega} (\pmb{\xi})\le \frac{s-1}{s} <H_{d-1,s}$.

\vfil\eject
\noindent {\bf Case 2.}
$ f (\pmb{\alpha}_{\nu})\neq 1 $ for all $\nu$ large enough.
\smallskip

For such $\nu$ the difference
$f (\pmb{\alpha}_\nu )-1$ is a nonzero rational number with denominator $Dq_{\nu}^s$. {Therefore}
 we have 
 $$| f (\pmb{\alpha}_\nu )-1| \ge\frac{1}{Dq_\nu^s}.
 $$
 Now  from (\ref{6qqq}) we see that
\begin{equation}\label{c}
\frac{1}{Dq_\nu^s}\le 
 |f (\pmb{\alpha}_\nu )-1| =
 |f (\pmb{\alpha}_\nu )  -  f (\pmb{\xi} )|=
 O\left(
\frac{\zeta_\nu}{q_\nu} \right)
\end{equation}
and
\begin{equation}\label{5qqq}
\zeta_\nu \ge \frac{c_2}{q_\nu^{s-1}}
\end{equation}
with some positive constant $c_2$ depending on $f$ and $\pmb{\xi}$. As this inequality holds for all large $\nu$, we conclude that
$\omega (\pmb{\xi}) \le s-1$
and
$$
            \frac{s-1}{\hat{\omega} (\pmb{\xi}) }\ge \frac{\omega (\pmb{\xi}) }{\hat{\omega} (\pmb{\xi}) }\ge G_d.
  $$
  Recall that $G_d$ is a root of equation (\ref{root}).
This means that the upper bound for $\hat{\omega} (\pmb{\xi})$ is given by the unique positive root of the equation 
$$
\left(\frac{s-1}{x}\right)^{d-1} =
   \frac{x}{1-x} \left( \left(\frac{s-1}{x}\right)^{d-2} + \left(\frac{s-1}{x}\right)^{d-3} + \dots+ 
   \frac{s-1}{x}
   +1\right)
   $$
which coincide{s} with
(\ref{0qqq}) {if} $d$ {is} replaced by $d-1$. \qed
    
       \medskip

    {We close the section by observing that  the  argument used in the proof of Case 2 above does not rely on the homogeneity of $f$. Thus the following result can be established.}
\begin{customthm}{2b}\label{2b}   {Suppose that
$f$ is an {\it arbitrary} polynomial of degree $s$ in $d$ variables  with rational coefficients 
 such that 
$$\#\{\pmb{x}\in\Q^d: f(\pmb{x} )= 1\} < \infty,$$
and let $\pmb{\xi} = (\xi_1,...,\xi_d)\notin\Q^d$ be such that
$f(\pmb{\xi} )= 1$.
Then: }
\smallskip

 \noindent  {{\bf (i)} $\omega(\pmb{\xi}) \le s-1$;}
  
\smallskip
       \noindent  {{\bf (ii)}
if  $\pmb{\xi}$ is totally irrational, then $\hat{\omega}(\pmb{\xi}) \le H_{d-1,s}$.}
\end{customthm}
      
{The proof is left to the reader. In particular, the conclusion of Theorem \ref{2b} holds when $\{f = 1\}$ is an algebraic curve over $\Q$ of genus at least $2$, such as the one mentioned in Remark~\ref{3}.}

  \ignore
  {{\begin{proof} We follow the lines of Case 2 above. If   $
   \pmb{a} \in \mathbb{Z}^{d}
   $, $q\in\N$ and  $D $ is the common denominator of all the coefficients of $f$, we have 
 $$\left|f \left(\frac{\pmb{a}}q  \right)-1\right| \ge\frac{1}{Dq^s}
 $$ if $q$ is large enough. On the other hand,
 $$
 \left|f \left(\frac{\pmb{a}}q  \right)-1\right| =
\left|f \left(\frac{\pmb{a}}q  \right)-f (\pmb{\xi} )\right| =
 O\left(\frac{\pmb{|}q\pmb{a}  - \pmb{\xi}\pmb{|}}q \right).$$
Hence $$\pmb{|}q\pmb{a}  - \pmb{\xi}\pmb{|} \ge \frac{c_3}{q^{s-1}}$$ with some positive constant $c_3$ depending on $f$ and $\pmb{\xi}$, and {\bf (i)} follows. The derivation of  {\bf (ii)} is identical to the last part of the proof of Case 2.
 \end{proof}}}
 
\section{ Proof of Theorem \ref{1a}}\label{t4}

The proof of Theorem \ref{1a} is similar to the proof of Theorem \ref{2a}.
We take $m = d+1$ and consider a sequence of best simultaneous approximation vectors 
   $$
   \pmb{z}_\nu = (q_\nu, a_{1,\nu},\dots,a_{d,\nu}, A_\nu) \in \mathbb{Z}^{d+2}, \,\,\,\, {\nu\in\N},
   $$
of  $\Theta = \Xi_f$ as in \eqref{ooooW},
   and the corresponding  distances from $q_\nu\Xi_f$ to $\Z^{d+1}$: 
   $$
   \zeta_\nu=\max\big(
   |q_\nu\xi_1 - a_{1,\nu}|,\dots,
   |q_\nu\xi_d - a_{d,\nu}|
   ,
   |q_\nu f (\pmb{\xi}) - A_\nu|\big) 
   .
   $$
   We also need ``{shortened}" rational approximation vectors
   $$
   \pmb{\alpha}_\nu =\left(  \frac{a_{1,\nu}}{q_\nu},\dots,\frac{ a_{d,\nu}}{q_\nu}\right)\in \mathbb{Q}^{d}.   $$
   {Note that} now it may happen that  
   \begin{equation}\label{anu}
     \pmb{\alpha}_{\nu -1} =   \pmb{\alpha}_\nu 
   \end{equation}
   for some $\nu$.
   
\begin{lemma}\label{l2}
   Suppose that 
     \eqref{anu} holds and
     \begin{equation}\label{dell}
     f(\pmb{\alpha}_\nu) =\frac{A_\nu}{q_\nu }.
     \end{equation}
   Then
   \begin{equation}\label{deelt}
   \Delta = {\rm g.c.d.}  (q_\nu, a_{1,\nu},\dots,a_{d,\nu}) =
   O\left( q_\nu^{\frac{s-1}{s}}\right).
   \end{equation}
   \end{lemma}
   
   \begin{proof}   We know that
   $$
   {\rm g.c.d.}  (q_\nu, a_{1,\nu},\dots,a_{d,\nu}, A_\nu)  =1
   $$
   and {thus} 
   $$
   {\rm g.c.d.} (\Delta, A_\nu) = 1.
   $$
   From (\ref{dell}) we see that
   $$
   DA_\nu q_\nu^{s-1} = D q_\nu^s f(\pmb{\alpha}_\nu)\in \mathbb{Z}.
   $$
   But $\Delta \,| \, a_{j,\nu} $ for any $j$. As $D q_\nu^s f\left(  \frac{\cdot}{q_\nu}\right)$ is a homogeneous polynomial of degree $s$ with integer coefficients, we deduce that
   $$
   \Delta^s \,| \, Dq_{\nu}^{s-1}.
   $$
   This gives (\ref{deelt}).
   \end{proof}

   To prove Theorem \ref{1a} we consider three cases.
   
   \medskip
   \noindent {\bf Case 1.1.}    For infinitely many $\nu$ {(\ref{anu}) and (\ref{dell}) hold}.
   \smallskip
   
In this case the vectors
$$
(q_{\nu-1}, a_{1,\nu-1},\dots,a_{d,\nu-1}),\,\,\,\,\,
 (q_\nu, a_{1,\nu},\dots,a_{d,\nu})
 $$
 are proportional, but the vectors
$$
(q_{\nu-1}, a_{1,\nu-1},\dots,a_{d,\nu-1},A_{\nu-1}),\,\,\,\,\,
 (q_\nu, a_{1,\nu},\dots,a_{d,\nu}, A_\nu)
 $$
 are not proportional. This means that
\begin{equation}\label{ppop}
\left|
\begin{array}{cc}
q_{\nu-1} &A_{\nu-1}
\cr
q_{\nu} &A_{\nu}
\end{array}
\right| \neq 0
.
\end{equation}
There exists a primitive vector
$$
(q_*, a_{1,*},\dots,a_{d,*})\in\mathbb{Z}^{d+1},\,\,\,\,\,    {\rm g.c.d.} (q_*, a_{1,*},\dots,a_{d,*})= 1, \,\,\,\,\, q_*\ge 1,
$$
such that 
$
 (q_\nu, a_{1,\nu},\dots,a_{d,\nu}) = \Delta \cdot
 (q_*, a_{1,*},\dots,a_{d,*})$ and $
 (q_{\nu-1}, a_{1,\nu-1},\dots,a_{d,\nu-1}) = \Delta'\cdot
 (q_*, a_{1,*},\dots,a_{d,*})
 $,
 where
 $$
  \Delta = {\rm g.c.d.}  (q_\nu, a_{1,\nu},\dots,a_{d,\nu}) ,\,\,\,\,\,\,
   \Delta' = {\rm g.c.d.}  (q_{\nu-1}, a_{1,\nu-1},\dots,a_{d,\nu-1}) .
   $$
 In particular
 $$
 q_\nu = \Delta q_*,\,\,\,\,\, q_{\nu-1} =\Delta ' q_*
 $$
 and
 $$
 \left|
\begin{array}{cc}
q_{\nu-1} &A_{\nu-1}
\cr
q_{\nu} &A_{\nu}
\end{array}
\right| \equiv 0\pmod{q_*}.
 $$
 Now from (\ref{ppop}) we deduce
 $$
\frac{q_\nu}{\Delta} = q_*\le
|\,
\left|
\begin{array}{cc}
q_{\nu-1} &A_{\nu-1}
\cr
q_{\nu} &A_{\nu}
\end{array}
\right|\,|
\le 2 q_\nu |q_{\nu-1} f(\pmb{\xi}) - A_{\nu-1}| \le 2q_\nu \zeta_{\nu-1}\le 2q_\nu ^{1-\gamma}
$$
by (\ref{ww}). Thus we get
$$
q_\nu^\gamma \le 2\Delta.
$$
We apply Lemma \ref{2} to see that
$
\gamma \le \frac{s-1}{s}
,$
and hence $ \hat{\omega} (\Xi_f) \le \frac{s-1}{s} < H_{d,s}$.

   \medskip
   \noindent {\bf Case 1.2.}  For infinitely many $\nu$  {(\ref{dell}) holds}  with $\pmb{\alpha}_{\nu-1}\neq\pmb{\alpha}_\nu$.
   \smallskip

We proceed  similarly to   Case 1 from the proof of Theorem \ref{2a} by 
 applying Lemma \ref{l1}{\bf (i)}
with $ A=q_\nu$ for $\pmb{\alpha} = \pmb{\alpha}_\nu \neq\pmb{\beta}=\pmb{\alpha}_{\nu-1}$.
From (\ref{oo}),  similarly to (\ref{c1}), {for large enough $\nu$} we get  
$$
 \frac{1}{(DK)^{\frac{1}{s}}q_{\nu}^{\frac{s-1}{s} }q_{\nu-1}}\le
  \sqrt{d}  \max_{1\le k\le d}
 \left|
 \frac{a_{k,\nu}}{q_\nu}-
  \frac{a_{k,\nu-1}}{q_{\nu-1}}
 \right|
 \le
  \frac{2 \sqrt{d}  }{q_{\nu-1}q_\nu^\gamma}.
$$
Again
 $ q_\nu^\gamma = O\left( q_\nu^{\frac{s-1}{s}}\right)$ for infinitely many $\nu$, and
 $\hat{\omega} (\Xi_f)\le \frac{s-1}{s} <H_{d,s}$.

   \medskip

   \noindent {\bf Case 2.}  
$ f (\pmb{\alpha}_{\nu}) \neq\frac{A_\nu}{q_\nu}  $ for all $\nu$ large enough.
   \smallskip

This case is similar to   Case 2 from the proof of Theorem \ref{2a}.
Now the difference
$f (\pmb{\alpha}_\nu )-\frac{A_\nu}{q_\nu}$ is a nonzero rational number with denominator $Dq_{\nu}^s$. Therefore
 we have 
 $$\left| f (\pmb{\alpha}_\nu )-\frac{A_\nu}{q_\nu}\right| \ge\frac{1}{Dq_\nu^s}.
 $$
 Analogously to (\ref{c}) we now get  (\ref{5qqq}), {which} leads to 
$\omega (\Xi_f) \le s-1$.
{Then, applying Theorem~\ref{A} to the $(d+1)$-dimensional vector
  $\Xi_f$, we obtain}
$$
            \frac{s-1}{\hat{\omega} (\Xi_f) }\ge \frac{\omega (\Xi_f) }{\hat{\omega} (\Xi_f) }\ge G_{d+1}.
  $$
  This shows that the positive root of equation
  $$
\left(\frac{s-1}{x}\right)^{d} =
   \frac{x}{1-x} \left( \left(\frac{s-1}{x}\right)^{d-1} + \left(\frac{s-1}{x}\right)^{d-2} + \dots+ 
   \frac{s-1}{x}
   +1\right)
   $$
   gives {an} upper bound for $\hat{\omega}(\Xi_f)$. Theorem \ref{1a} is proved. 
    \qed

    \section{ Proof of Theorem 3a}\label{t6}    
    Let $$q_1 < q_2< \,\dots\, <q_\nu<q_{\nu+1}<\dots$$ be the   sequence of points where the function 
    $\Psi_{f, \pmb{\xi}} (T)$
    is not continuous. 
Without loss of generality we may suppose that this sequence  is infinite.
    We consider the corresponding best approximation vectors
    $$
    \pmb{\alpha}_\nu = \left( \frac{a_{1,\nu}}{q_\nu},\dots,\frac{a_{d,\nu}}{q_\nu}\right) \in \mathbb{Q}^{d},\,\,\,\,\,\,\,
    $$
    where $a_{j,\nu}$ realize the minima {in} the definition of the function $\Psi_{f, \pmb{\xi}} (T)$.
    They satisfy
    $
 f( \pmb{\alpha}_\nu) = 1.
$
    By definition of the function $\Psi_{\pmb{\xi}}(T)$ and numbers $q_\nu$ we see that
    $$
    \Psi_{\pmb{\xi}}(T) = \Psi_{\pmb{\xi}}(q_{\nu-1})\,\,\,\, \text{for}\,\,\,\, q_{\nu-1} \le T < q_\nu.
    $$
    Now, since
    $\pmb{\alpha}_\nu\neq\pmb{\alpha}_{\nu-1}$,
we may apply Lemma \ref{l1}{\bf (ii)}
with $ A=q_\nu$ {and} $ B = q_{\nu-1}$.
         {Indeed, $\pmb{\alpha} = \pmb{\alpha}_\nu$ and $\pmb{\beta} =\pmb{\alpha}_{\nu-1}$  satisfy \eqref{zzz}  for large enough $\nu$, and  from} (\ref{ooooo})  we have  
  $$
    \frac{1}{DKq_{\nu-1}^{s-1}q_\nu^{s-1}}\le
     \pmb{|} \pmb{\alpha}_{\nu-1}- \pmb{\alpha}_\nu\pmb{|}^s\le
    (
    \pmb{|} \pmb{\alpha}_{\nu-1} - \pmb{\xi}\pmb{|}
 +
 \pmb{|} \pmb{\alpha}_\nu - \pmb{\xi}\pmb{|})^s
 \le
 2^{s}
\pmb{|} \pmb{\alpha}_{\nu-1} - \pmb{\xi}\pmb{|}^s ,
    $$
    since $ \pmb{|} \pmb{\alpha}_\nu - \pmb{\xi}\pmb{|}\le \pmb{|} \pmb{\alpha}_{\nu-1} - \pmb{\xi}\pmb{|}$.
 {Thus}
         $$
    \frac{1}{2^{s}D{K}}
    \le  q_{\nu}^{s-1}q_{\nu-1}^{s-1}  \pmb{|} \pmb{\alpha}_{\nu-1} - \pmb{\xi}\pmb{|}^s
    \le q_\nu^{s-1} \Psi_{f,\pmb{\xi}} (q_{\nu-1})  .
    $$

    This means that
    $$
    \lim_{T\to q_\nu -0 }  T^{s-1}\cdot \Psi_{f,\pmb{\xi}} (T) =   q_\nu^{s-1}\Psi_{f,\pmb{\xi}} (q_{\nu-1})   \ge \frac{1}{2^sD{K}}\,.
    $$
     Theorem \ref{3a} is proved. \qed

\subsection*{Acknowledgements}

The authors were supported by NSF grant DMS-1600814 and RFBR grant No.\ 18-01-00886 respectively. This work was started during the second-named author's visit to Brandeis University, whose hospitality is gratefully acknowledged.  Thanks are also due to David Simmons, Barak Weiss and the anonymous referee for useful comments.

\end{document}